\documentclass[10.4pt]{amsart}
\usepackage{amssymb,amsmath,amsfonts,latexsym,amsthm,geometry}
\usepackage{amsmath}
\usepackage{amsfonts}
\usepackage{amssymb}
\usepackage{graphicx}

\providecommand{\U}[1]{\protect\rule{.1in}{.1in}}
\providecommand{\U}[1]{\protect\rule{.1in}{.1in}}
\providecommand{\U}[1]{\protect\rule{.1in}{.1in}}
\newtheorem{theorem}{Theorem}[section]
\newtheorem{corollary}[theorem]{Corollary}
\newtheorem{proposition}[theorem]{Proposition}
\newtheorem{lemma}[theorem]{Lemma}

\theoremstyle{definition}

\newtheorem{example}[theorem]{Example}

\begin{document}
\title[On the mixed $\left( \ell _{1},\ell _{2}\right) $-Littlewood
inequality and applications]{On the mixed $\left( \ell _{1},\ell _{2}\right)
$-Littlewood inequality for real scalars and applications}
\author[Daniel Pellegrino]{Daniel Pellegrino}
\address{Departamento de Matem\'{a}tica, Universidade Federal da Para\'{\i}%
ba, 58.051-900 - Jo\~{a}o Pessoa, Brazil.}
\email{pellegrino@pq.cnpq.br and dmpellegrino@gmail.com}
\author[Diana Serrano-Rodr\'{\i}guez]{Diana M. Serrano-Rodr\'{\i}guez}
\address{Departamento de Matem\'{a}tica, Universidade Federal de Pernambuco,
Recife, Brazil.}
\thanks{D. Pellegrino is supported by CNPq}

\begin{abstract}
In this paper we obtain the sharp estimates for the mixed $\left( \ell
_{1},\ell _{2}\right) $-Littlewood inequality for real scalars with
exponents $\left( 2,1,2,2....,2\right) .$ These results are applied to find
sharp estimates for the constants of a family of $3$-linear
Bohnenblust--Hille inequalities with multiple exponents.
\end{abstract}

\maketitle


\section{Introduction}

The mixed $\left( \ell _{1},\ell _{2}\right) $-Littlewood inequality for
real scalars asserts that for all continuous real $m$-linear forms $%
U:c_{0}\times \cdots \times c_{0}\rightarrow \mathbb{R}$ we have%
\begin{equation}
\sum_{j_{1}=1}^{N}\left( \sum_{j_{2},...,j_{m}=1}^{N}\left\vert
U(e_{j_{1}},...,e_{j_{m}})\right\vert ^{2}\right) ^{\frac{1}{2}}\leq \left(
\sqrt{2}\right) ^{m-1}\left\Vert U\right\Vert ,  \label{u8}
\end{equation}%
for all positive integers $N.$ From this inequality, using the intrinsic
symmetry of the context, it is not difficult, using a Minkowski-type
inequality, to prove that in fact for each $k\in \{2,...,m\}$ we have%
\begin{equation}
\left( \sum_{j_{1},...,j_{k-1}=1}^{N}\left( \sum_{j_{k}=1}^{N}\left(
\sum_{j_{k+1},...,j_{m}=1}^{N}\left\vert
U(e_{j_{1}},...,e_{j_{m}})\right\vert ^{2}\right) ^{\frac{1}{2}\times
1}\right) ^{\frac{1}{1}\times 2}\right) ^{\frac{1}{2}}\leq \left( \sqrt{2}%
\right) ^{m-1}\left\Vert U\right\Vert ,  \label{0009}
\end{equation}%
which is also called mixed $\left( \ell _{1},\ell _{2}\right) $-Littlewood
inequality. When we replace $\mathbb{R}$ by $\mathbb{C}$ it is well known
that $\left( \sqrt{2}\right) ^{m-1}$ can be replaced by $\left( 2/\sqrt{\pi }%
\right) ^{m-1}.$ In a more simplified notation we have inequalities with
\textquotedblleft multiple\textquotedblright\ exponents $\left(
1,2,2,....,2\right) ,...,\left( 2,....,2,1\right) $. These inequalities are
a crucial tool to prove the Bohnenblust--Hille inequality for multilinear
forms. Recall that the multilinear Bohnenblust--Hille inequality (\cite{bh})
for $\mathbb{K}=\mathbb{R}$ or $\mathbb{C}$ asserts that there exists a
sequence of positive scalars $\left( B_{m}^{\mathbb{K}}\right)
_{m=1}^{\infty }$ in $[1,\infty )$ such that
\begin{equation}
\left( \sum\limits_{i_{1},\ldots ,i_{m}=1}^{N}\left\vert
U(e_{i_{^{1}}},\ldots ,e_{i_{m}})\right\vert ^{\frac{2m}{m+1}}\right) ^{%
\frac{m+1}{2m}}\leq B_{m}^{\mathbb{K}}\sup_{\left\{ \left\Vert
z_{j}\right\Vert =1:j=1,...,m\right\} }\left\vert U(z_{1},\ldots
,z_{m})\right\vert  \label{ul}
\end{equation}%
for all continuous $m$-linear forms $U:c_{0}\times \cdots \times
c_{0}\rightarrow \mathbb{K}$ and every positive integer $N$, where $\left(
e_{i}\right) _{i=1}^{\infty }$ denotes the sequence of canonical vectors of $%
c_{0}.$ The Bohnenblust--Hille inequality can be seen as a predecessor of
the multilinear theory of absolutely summing operators (see, for instance,
\cite{botelho, matos, popa, pilar} and the references therein).

The connections between the mixed $\left( \ell _{1},\ell _{2}\right) $%
-Littlewood inequality and the Bohnenblust--Hille inequality are well-known
and can be easily explained with the interpolative approach from \cite[%
Section 2]{abps}. To obtain the Bohnenblust--Hille inequality from the mixed
$\left( \ell _{1},\ell _{2}\right) $-Littlewood inequalities it suffices to
observe that the exponent $\frac{2m}{m+1}$ can be seen as a multiple
exponent $\left( \frac{2m}{m+1},...,\frac{2m}{m+1}\right) $ and this
exponent is precisely the interpolation of the exponents%
\begin{equation}
\left( 1,2,2,....,2\right) ,...,\left( 2,....,2,1\right)  \label{87}
\end{equation}%
with $\theta _{1}=\cdots =\theta _{m}=1/m.$

It was recently proved in \cite{natal} that for real scalars the values $%
\left( \sqrt{2}\right) ^{m-1}$ are the sharp constants for (\ref{u8}).
However, the proof of \cite{natal} cannot be straightforwardly extended to
the general family of inequalities (\ref{0009}). Of course, a natural
question, whose answer can be useful in other inequalities, is whether the
upper estimates $\left( \sqrt{2}\right) ^{m-1}$ can be improved as long as
the exponent $1$ moves from the left to the right in (\ref{87}). In this
paper, among other results we show that for the multiple exponent $\left(
2,1,2,2,...,2\right) $ the sharp constants are still $\left( \sqrt{2}\right)
^{m-1}.$ It is worth mentioning that our approach seems to be not effective
to cover all the remaining cases $\left( 2,2,1,2,...,2\right) ,...,\left(
2,2,...,2,1\right) $ and this some open problems remain waiting for an
answer.

The exact values for the optimal constants $B_{m}^{\mathbb{K}}$ satisfying (%
\ref{ul}) are still unknown, although many progresses have been made in the
last few years. Having nice estimates for these constants is, in general,
crucial for applications (for instance in Quantum Information Theory (see
\cite{Mont}), and Complex Analysis \cite{bps}). The first estimates for $%
\left( B_{m}^{\mathbb{K}}\right) _{m=1}^{\infty }$ (\cite{bh, davie, kaiser,
quef}) suggested an exponential growth but only few years ago very different
results have appeared. In fact, for $m\geq 2$ the most recent estimates for
the optimal values for the constants satisfying (\ref{ul}) show a sublinear
growth:%
\begin{equation}
\begin{tabular}{lll}
$B_{m}^{\mathbb{C}}$ & $<$ & $m^{0.21392},$ and \\
$B_{m}^{\mathbb{R}}$ & $<$ & $1.3\cdot m^{0.36481}.$%
\end{tabular}
\label{77o}
\end{equation}%
More specifically, for complex scalars (see \cite{bps}),
\begin{equation*}
1\leq B_{m}^{\mathbb{C}}\leq \prod_{j=2}^{m}\Gamma \left( 2-\frac{1}{j}%
\right) ^{\frac{j}{2-2j}},
\end{equation*}%
where $\Gamma $ denotes the gamma function. For real scalars (see \cite{bps,
diniz}),%
\begin{equation*}
2^{1-\frac{1}{m}}\leq B_{m}^{\mathbb{R}}\leq \prod_{j=2}^{m}A_{\frac{2j-2}{j}%
}^{-1},
\end{equation*}%
where the constants $A_{p}$ denote the best constants satisfying Khinchine's
inequality (see \cite{Ha}), which are given by
\begin{equation*}
A_{p}:=\sqrt{2}\left( \frac{\Gamma (\frac{p+1}{2})}{\sqrt{\pi }}\right)
^{1/p},
\end{equation*}%
for $p>p_{0}\approx 1.85$ and

\begin{equation*}
A_{p}:=2^{\frac{1}{2}-\frac{1}{p}}
\end{equation*}%
for $p\leq p_{0}\approx 1.85.$ More precisely, the number $p_{0}$ $\in (1,2)$
is the solution of the following equality

\begin{equation*}
\Gamma \left( \frac{p_{0}+1}{2}\right) =\frac{\sqrt{\pi }}{2}.
\end{equation*}%
It is still an open problem, for real scalars, if the optimal constants $%
B_{m}^{\mathbb{R}}$ are $2^{1-\frac{1}{m}}$ or $\prod_{j=2}^{m}A_{\frac{2j-2%
}{j}}^{-1}$ or whether they lie strictly between these bounds. The only
known exact value appears in the real bilinear case, since $2^{1-\frac{1}{2}%
}=B_{2}^{\mathbb{R}}$ (see \cite{diniz})$.$ For the complex case, similar
questions remain open.

\bigskip In \cite{abps} it is proved that the Bohnenblust--Hille inequality
is a very particular case of a large family of sharp inequalities. More
precisely, the following general result was proved in \cite[Theorem 1.1]%
{abps}:

\begin{theorem}[Generalized Bohnenblust--Hille inequality, \protect\cite%
{abps}]
\label{bh_gen} Let $m\geq 2$ be a positive integer and let \linebreak $%
\mathbf{q}:=\left( q_{1},\dots ,q_{m}\right) \in \lbrack 1,2]^{m}.$ The
following assertions are equivalent:

(1) The sequence $\left( q_{1},\dots ,q_{m}\right) $ satisfies
\begin{equation*}
\frac{1}{q_{1}}+\cdots +\frac{1}{q_{m}}\leq \frac{m+1}{2}.
\end{equation*}%
(2) There exists a constant $C_{m,\left( q_{1},\dots ,q_{m}\right) }^{%
\mathbb{K}}\geq 1$ such that
\begin{equation*}
\left( \sum_{j_{1}=1}^{N}\left( \sum_{j_{2}=1}^{N}\left( \cdots \left(
\sum_{j_{m}=1}^{N}\left\vert T(e_{j_{1}},...,e_{j_{m}})\right\vert
^{q_{m}}\right) ^{\frac{q_{m-1}}{q_{m}}}\cdots \right) ^{\frac{q_{2}}{q_{3}}%
}\right) ^{\frac{q_{1}}{q_{2}}}\right) ^{\frac{1}{q_{1}}}\leq C_{m,\mathbf{q}%
}^{\mathbb{K}}\left\Vert T\right\Vert
\end{equation*}%
for all continuous $m$-linear forms $T:c_{0}\times \cdots \times
c_{0}\rightarrow \mathbb{K}$, and every positive integer $N$.
\end{theorem}

Observe that the Bohnenblust--Hille inequality is curiously the particular
case%
\begin{equation*}
q_{1}=q_{2}=...=q_{m}=\frac{2m}{m+1}.
\end{equation*}%
In the recent years, some works have provided optimal estimates for some
particular constants $C_{m,\mathbf{q}}^{\mathbb{R}}$. For instance, $C_{2,%
\mathbf{q}}^{\mathbb{R}}=2^{\frac{1}{2}}$, (see \cite{abps}) and $C_{m,%
\mathbf{q}}^{\mathbb{R}}=2^{\frac{m-1}{2}}$, for $\mathbf{q}=\left(
1,2,...,2\right) $ and all $m\geq 2$, (see \cite{natal}).

\bigskip

This paper is organized as follows. In Section 2 and Section 3 we state and
prove our first main results related to the mixed $\left( \ell _{1},\ell
_{2}\right) $-Littlewood inequalities. In Section 4 we prove some estimates
for the upper bounds of the generalized Bohnenblust--Hille inequality and
the results of the sections 2,3,4 are used in the final section to obtain
sharp estimates for the generalized Bohnenblust--Hille inequality for
certain $3$-linear forms.

\section{First main result}

\bigskip For $m\geq 2$ and $1\leq i\leq m$, we define
\begin{equation*}
\widehat{q}_{i,m}:=\frac{q_{1}q_{2}...q_{m}}{q_{i}}.
\end{equation*}%
Our first main result is the following theorem that extends the main result
of \cite{natal}, as we shall see in (\ref{gggg7}) and (\ref{gggg8}):

\begin{theorem}
\label{general} If\textit{\ }$m\geq 2$ is a positive integer, \textit{and} $%
\mathbf{q}:=(q_{1},...,q_{m})\in \left[ 1,2\right] ^{m}$ are \textit{such
that}
\begin{equation*}
\frac{1}{q_{1}}+\cdots +\frac{1}{q_{m}}=\frac{m+1}{2},
\end{equation*}%
t\textit{hen there exists a constant}
\begin{equation}
C_{m,\mathbf{q}}^{\mathbb{R}}\geq 2^{\frac{\left( m-1\right) \widehat{q}%
_{1,m}+\left( \sum_{i=2}^{m}\widehat{q}_{i,m}\right) -\left( m-1\right)
q_{1}q_{2}...q_{m}}{q_{1}q_{2}...q_{m}}}  \label{lower}
\end{equation}%
\textit{such that}
\begin{equation*}
\left( \sum_{j_{1}=1}^{n}\left( \sum_{j_{2}=1}^{n}\left( \cdots \left(
\sum_{j_{m}=1}^{n}\left\vert T(e_{j_{1}},...,e_{j_{m}})\right\vert
^{q_{m}}\right) ^{\frac{q_{m-1}}{q_{m}}}\cdots \right) ^{\frac{q_{2}}{q_{3}}%
}\right) ^{\frac{q_{1}}{q_{2}}}\right) ^{\frac{1}{q_{1}}}\leq C_{m,\mathbf{q}%
}^{\mathbb{R}}\left\Vert T\right\Vert
\end{equation*}%
\textit{for all continuous }$m$\textit{--linear forms }$T:c_{0}\times \cdots
\times c_{0}\rightarrow \mathbb{R}$\textit{\ and all positive integers }$n.$
\end{theorem}

\begin{proof}
\bigskip In \cite{diniz}, for all positive integers $m\geq 2$, the $m$%
-linear forms $T_{m}$ are inductively defined as
\begin{equation*}
\begin{tabular}{llll}
$T_{2}:$ & $\ell _{\infty }^{2}\times \ell _{\infty }^{2}$ & $\rightarrow $
& $\mathbb{R}$ \\
& $\left( x^{\left( 1\right) },x^{\left( 2\right) }\right) $ & $\mapsto $ & $%
x_{1}^{\left( 1\right) }x_{1}^{\left( 2\right) }+x_{1}^{\left( 1\right)
}x_{2}^{\left( 2\right) }+x_{2}^{\left( 1\right) }x_{1}^{\left( 2\right)
}-x_{2}^{\left( 1\right) }x_{2}^{\left( 2\right) }$%
\end{tabular}%
\end{equation*}%
and%
\begin{equation*}
\begin{tabular}{llll}
$T_{m}:$ & $\ell _{\infty }^{2^{m-1}}\times \cdots \times \ell _{\infty
}^{2^{m-1}}$ & $\rightarrow $ & $\mathbb{R}$ \\
& $\left( x^{\left( 1\right) },...,x^{\left( m\right) }\right) $ & $\mapsto $
& $\left( x_{1}^{\left( m\right) }+x_{2}^{\left( m\right) }\right)
T_{m-1}\left( x^{\left( 1\right) },...,x^{\left( m-1\right) }\right) $ \\
&  &  & $+\left( x_{1}^{\left( m\right) }-x_{2}^{\left( m\right) }\right)
T_{m-1}\left( \widehat{x}^{\left( 1\right) },\widehat{x}^{\left( 2\right)
},...,\widehat{x}^{\left( m-1\right) }\right) ,$%
\end{tabular}%
\end{equation*}%
where, for $1\leq k\leq m$,
\begin{equation*}
x^{\left( k\right) }=\left( x_{j}^{\left( k\right) }\right)
_{j=1}^{2^{m-1}}\in \ell _{\infty }^{2^{m-1}},
\end{equation*}
and%
\begin{eqnarray*}
\widehat{x}^{\left( 1\right) } &=&B^{2^{m-2}}\left( x^{\left( 1\right)
}\right) , \\
\widehat{x}^{\left( 2\right) } &=&B^{2^{m-2}}\left( x^{\left( 2\right)
}\right) , \\
\widehat{x}^{\left( 3\right) } &=&B^{2^{m-3}}\left( x^{\left( 3\right)
}\right) , \\
\widehat{x}^{\left( 4\right) } &=&B^{2^{m-4}}\left( x^{\left( 4\right)
}\right) , \\
&&\vdots \\
\widehat{x}^{\left( m-2\right) } &=&B^{2^{2}}\left( x^{\left( m-2\right)
}\right) , \\
\widehat{x}^{\left( m-1\right) } &=&B^{2}\left( x^{\left( m-1\right)
}\right) ,
\end{eqnarray*}%
where $B$ is the backward shift operator in $\ell _{p}^{2^{m-1}}.$ As a
matter of fact, we can observe that the domain of $T_{m}$ can be chosen as $%
\ell _{\infty }^{2^{m-1}}\times \ell _{\infty }^{2^{m-1}}\times \ell
_{\infty }^{2^{m-2}}\times \ell _{\infty }^{2^{m-3}}\times \cdots \times
\ell _{\infty }^{2^{2}}\times \ell _{\infty }^{2}$.

Let us see that for $m\geq 2$ and $1\leq i\leq m$ we have%
\begin{equation}
\left( \sum_{j_{1}=1}^{2^{m-1}}\left( \sum_{j_{2}=1}^{2^{m-1}}\left( \cdots
\left( \sum_{j_{m}=1}^{2^{m-1}}\left\vert
T_{m}(e_{j_{1}},...,e_{j_{m}})\right\vert ^{q_{m}}\right) ^{\frac{q_{m-1}}{%
q_{m}}}\cdots \right) ^{\frac{q_{2}}{q_{3}}}\right) ^{\frac{q_{1}}{q_{2}}%
}\right) ^{\frac{1}{q_{1}}}=2^{\frac{\left( m-1\right) \widehat{q}%
_{1,m}+\sum_{i=2}^{m}\widehat{q}_{i,m}}{q_{1}q_{2}...q_{m}}}.  \label{pppp0}
\end{equation}%
In fact, for $m=2$ it is immediate since\
\begin{equation*}
\left( \sum_{j_{1}=1}^{2}\left( \sum_{j_{2}=1}^{2}\left\vert
T_{2}(e_{j_{1}},e_{j_{2}})\right\vert ^{q_{2}}\right) ^{\frac{q_{1}}{q_{2}}%
}\right) ^{\frac{1}{q_{1}}}=2^{\frac{1}{q_{1}}+\frac{1}{q_{2}}}=2^{\frac{%
q_{2}+q_{1}}{q_{1}q_{2}}}.
\end{equation*}%
Let us prove by induction. Suppose that it is valid for $m-1$ and let us
prove for $m.$ In other words we shall prove that if
\begin{eqnarray*}
&&\left( \sum_{j_{1}=1}^{2^{m-2}}\left( \sum_{j_{2}=1}^{2^{m-2}}\left(
\cdots \left( \sum_{j_{m-1}=1}^{2^{m-2}}\left\vert
T_{m-1}(e_{j_{1}},...,e_{j_{m-1}})\right\vert ^{q_{m-1}}\right) ^{\frac{%
q_{m-2}}{q_{m-1}}}\cdots \right) ^{\frac{q_{2}}{q_{3}}}\right) ^{\frac{q_{1}%
}{q_{2}}}\right) ^{\frac{1}{q_{1}}} \\
&=&2^{\frac{\left( m-2\right) \widehat{q}_{1,m-1}+\sum_{i=2}^{m-1}\widehat{q}%
_{i,m-1}}{q_{1}q_{2}...q_{m-1}}}
\end{eqnarray*}%
is valid, then
\begin{eqnarray*}
&&\left( \sum_{j_{1}=1}^{2^{m-1}}\left( \sum_{j_{2}=1}^{2^{m-1}}\left(
\cdots \left( \sum_{j_{m}=1}^{2^{m-1}}\left\vert
T_{m}(e_{j_{1}},...,e_{j_{m}})\right\vert ^{q_{m}}\right) ^{\frac{q_{m-1}}{%
q_{m}}}\cdots \right) ^{\frac{q_{2}}{q_{3}}}\right) ^{\frac{q_{1}}{q_{2}}%
}\right) ^{\frac{1}{q_{1}}} \\
&=&2^{\frac{\left( m-1\right) \widehat{q}_{1,m}+\sum_{i=2}^{m}\widehat{q}%
_{i,m}}{q_{1}q_{2}...q_{m}}}.
\end{eqnarray*}%
Note that \
\begin{eqnarray*}
&&\sum_{j_{m-1}=1}^{2^{m-1}}\left( \sum_{j_{m}=1}^{2^{m-1}}\left\vert
T_{m}(e_{j_{1}},...,e_{j_{m}})\right\vert ^{q_{m}}\right) ^{\frac{q_{m-1}}{%
q_{m}}} \\
&=&\sum_{j_{m-1}=1}^{2^{m-1}}\left( \left\vert
T_{m}(e_{j_{1}},...,e_{j_{m-1}},e_{1})\right\vert ^{q_{m}}+\left\vert
T_{m}\left( e_{j_{1}},...,e_{j_{m-1}},e_{2}\right) \right\vert
^{q_{m}}\right) ^{\frac{q_{m-1}}{q_{m}}} \\
&=&\sum_{j_{m-1}=1}^{2^{m-1}}\left(
\begin{array}{c}
\left\vert T_{m-1}(e_{j_{1}},...,e_{j_{m-1}})+T_{m-1}(B^{2^{m-2}}\left(
e_{j_{1}}\right) ,...,B^{2}\left( e_{j_{m-1}}\right) )\right\vert ^{q_{m}}
\\
+\left\vert T_{m-1}(e_{j_{1}},...,e_{j_{m-1}})-T_{m-1}(B^{2^{m-2}}\left(
e_{j_{1}}\right) ,...,B^{2}\left( e_{j_{m-1}}\right) )\right\vert ^{q_{m}}%
\end{array}%
\right) ^{\frac{q_{m-1}}{q_{m}}} \\
&&%
\begin{array}{c}
=\left(
\begin{array}{c}
\left\vert
T_{m-1}(e_{j_{1}},...,e_{j_{m-2}},e_{1})+T_{m-1}(B^{2^{m-2}}\left(
e_{j_{1}}\right) ,...,B^{2}\left( e_{1}\right) )\right\vert ^{q_{m}} \\
+\left\vert
T_{m-1}(e_{j_{1}},...,e_{j_{m-2}},e_{1})-T_{m-1}(B^{2^{m-2}}\left(
e_{j_{1}}\right) ,...,B^{2}\left( e_{1}\right) )\right\vert ^{q_{m}}%
\end{array}%
\right) ^{\frac{q_{m-1}}{q_{m}}} \\
+\left(
\begin{array}{c}
\left\vert
T_{m-1}(e_{j_{1}},...,e_{j_{m-2}},e_{2})+T_{m-1}(B^{2^{m-2}}\left(
e_{j_{1}}\right) ,...,B^{2}\left( e_{2}\right) )\right\vert ^{q_{m}} \\
+\left\vert
T_{m-1}(e_{j_{1}},...,e_{j_{m-2}},e_{2})-T_{m-1}(B^{2^{m-2}}\left(
e_{j_{1}}\right) ,...,B^{2}\left( e_{2}\right) )\right\vert ^{q_{m}}%
\end{array}%
\right) ^{\frac{q_{m-1}}{q_{m}}} \\
+\left(
\begin{array}{c}
\left\vert
T_{m-1}(e_{j_{1}},...,e_{j_{m-2}},e_{3})+T_{m-1}(B^{2^{m-2}}\left(
e_{j_{1}}\right) ,...,B^{2}\left( e_{3}\right) )\right\vert ^{q_{m}} \\
+\left\vert
T_{m-1}(e_{j_{1}},...,e_{j_{m-2}},e_{3})-T_{m-1}(B^{2^{m-2}}\left(
e_{j_{1}}\right) ,...,B^{2}\left( e_{3}\right) )\right\vert ^{q_{m}}%
\end{array}%
\right) ^{\frac{q_{m-1}}{q_{m}}} \\
+\left(
\begin{array}{c}
\left\vert
T_{m-1}(e_{j_{1}},...,e_{j_{m-2}},e_{4})+T_{m-1}(B^{2^{m-2}}\left(
e_{j_{1}}\right) ,...,B^{2}\left( e_{4}\right) )\right\vert ^{q_{m}} \\
+\left\vert
T_{m-1}(e_{j_{1}},...,e_{j_{m-2}},e_{4})-T_{m-1}(B^{2^{m-2}}\left(
e_{j_{1}}\right) ,...,B^{2}\left( e_{4}\right) )\right\vert ^{q_{m}}%
\end{array}%
\right) ^{\frac{q_{m-1}}{q_{m}}}%
\end{array}%
\end{eqnarray*}

\begin{eqnarray*}
&=&%
\begin{array}{c}
\left( 2\left\vert T_{m-1}(e_{j_{1}},...,e_{j_{m-2}},e_{1})\right\vert
^{q_{m}}\right) ^{\frac{q_{m-1}}{q_{m}}}+\left( 2\left\vert
T_{m-1}(e_{j_{1}},...,e_{j_{m-2}},e_{2})\right\vert ^{q_{m}}\right) ^{\frac{%
q_{m-1}}{q_{m}}}+ \\
\left( 2\left\vert T_{m-1}(B^{2^{m-2}}\left( e_{j_{1}}\right)
,...,B^{2}\left( e_{3}\right) )\right\vert ^{q_{m}}\right) ^{\frac{q_{m-1}}{%
q_{m}}}+\left( 2\left\vert T_{m-1}(B^{2^{m-2}}\left( e_{j_{1}}\right)
,...,B^{2}\left( e_{4}\right) )\right\vert ^{q_{m}}\right) ^{\frac{q_{m-1}}{%
q_{m}}}%
\end{array}
\\
&=&2^{\frac{q_{m-1}}{q_{m}}}\left( \sum_{j_{m-1}=1}^{2}\left\vert
T_{m-1}(e_{j_{1}},...,e_{j_{m-2}},e_{j_{m-1}})\right\vert
^{q_{m-1}}+\sum_{j_{m-1}=3}^{4}\left\vert T_{m-1}(B^{2^{m-2}}\left(
e_{j_{1}}\right) ,...,B^{2}\left( e_{j_{m-1}}\right) )\right\vert
^{q_{m-1}}\right) .
\end{eqnarray*}%
By making
\begin{equation*}
\sum_{j_{m-1}=1}^{2}\left\vert
T_{m-1}(e_{j_{1}},...,e_{j_{m-2}},e_{j_{m-1}})\right\vert ^{q_{m-1}}:=A_{1}
\end{equation*}%
and%
\begin{equation*}
\sum_{j_{m-1}=3}^{4}\left\vert T_{m-1}(B^{2^{m-2}}\left( e_{j_{1}}\right)
,...,B^{2}\left( e_{j_{m-1}}\right) )\right\vert ^{q_{m-1}}:=A_{2},
\end{equation*}%
and using the induction hypothesis it follows that
\begin{eqnarray*}
&&\left( \sum_{j_{1}=1}^{2^{m-1}}\left( \sum_{j_{2}=1}^{2^{m-1}}\left(
\cdots \left( \sum_{j_{m}=1}^{2^{m-1}}\left\vert
T_{m}(e_{j_{1}},...,e_{j_{m}})\right\vert ^{q_{m}}\right) ^{\frac{q_{m-1}}{%
q_{m}}}\cdots \right) ^{\frac{q_{2}}{q_{3}}}\right) ^{\frac{q_{1}}{q_{2}}%
}\right) ^{\frac{1}{q_{1}}} \\
&=&2^{\frac{1}{q_{m}}}\left( \sum_{j_{1}=1}^{2^{m-1}}\left(
\sum_{j_{2}=1}^{2^{m-1}}\left( \cdots \sum_{j_{m-2}=1}^{2^{m-1}}\left(
A_{1}+A_{2}\right) ^{\frac{q_{m-2}}{q_{m-1}}}\cdots \right) ^{\frac{q_{2}}{%
q_{3}}}\right) ^{\frac{q_{1}}{q_{2}}}\right) ^{\frac{1}{q_{1}}} \\
&=&2^{\frac{1}{q_{m}}}\left(
\begin{array}{c}
\sum_{j_{1}=1}^{2^{m-2}}\left( \sum_{j_{2}=1}^{2^{m-1}}\left( \cdots
\sum_{j_{m-2}=1}^{2^{m-1}}\left( A_{1}+A_{2}\right) ^{\frac{q_{m-2}}{q_{m-1}}%
}\cdots \right) ^{\frac{q_{2}}{q_{3}}}\right) ^{\frac{q_{1}}{q_{2}}} \\
+\sum_{j_{1}=2^{m-2}+1}^{2^{m-1}}\left( \sum_{j_{2}=1}^{2^{m-1}}\left(
\cdots \sum_{j_{m-2}=1}^{2^{m-1}}\left( A_{1}+A_{2}\right) ^{\frac{q_{m-2}}{%
q_{m-1}}}\cdots \right) ^{\frac{q_{2}}{q_{3}}}\right) ^{\frac{q_{1}}{q_{2}}}%
\end{array}%
\right) ^{\frac{1}{q_{1}}} \\
&=&2^{\frac{1}{q_{m}}}\left(
\begin{array}{c}
\sum_{j_{1}=1}^{2^{m-2}}\left( \sum_{j_{2}=1}^{2^{m-1}}\left( \cdots
\sum_{j_{m-2}=1}^{2^{m-1}}\left( A_{1}\right) ^{\frac{q_{m-2}}{q_{m-1}}%
}\cdots \right) ^{\frac{q_{2}}{q_{3}}}\right) ^{\frac{q_{1}}{q_{2}}} \\
+\sum_{j_{1}=2^{m-2}+1}^{2^{m-1}}\left( \sum_{j_{2}=1}^{2^{m-1}}\left(
\cdots \sum_{j_{m-2}=1}^{2^{m-1}}\left( A_{2}\right) ^{\frac{q_{m-2}}{q_{m-1}%
}}\cdots \right) ^{\frac{q_{2}}{q_{3}}}\right) ^{\frac{q_{1}}{q_{2}}}%
\end{array}%
\right) ^{\frac{1}{q_{1}}} \\
&=&2^{\frac{1}{q_{m}}}\left( 2\sum_{j_{1}=1}^{2^{m-2}}\left(
\sum_{j_{2}=1}^{2^{m-1}}\left( \cdots \sum_{j_{m-2}=1}^{2^{m-1}}\left(
A_{1}\right) ^{\frac{q_{m-2}}{q_{m-1}}}\cdots \right) ^{\frac{q_{2}}{q_{3}}%
}\right) ^{\frac{q_{1}}{q_{2}}}\right) ^{\frac{1}{q_{1}}} \\
&=&2^{\frac{1}{q_{m}}+\frac{1}{q_{1}}}\left( \sum_{j_{1}=1}^{2^{m-2}}\left(
\sum_{j_{2}=1}^{2^{m-1}}\left( \cdots \sum_{j_{m-2}=1}^{2^{m-1}}\left(
\sum_{j_{m-1}=1}^{2}\left\vert T_{m-1}(e_{j_{1}},...,e_{j_{m-1}})\right\vert
^{q_{m-1}}\right) ^{\frac{q_{m-2}}{q_{m-1}}}\cdots \right) ^{\frac{q_{2}}{%
q_{3}}}\right) ^{\frac{q_{1}}{q_{2}}}\right) ^{\frac{1}{q_{1}}} \\
&=&2^{\frac{1}{q_{m}}+\frac{1}{q_{1}}}2^{\frac{\left( m-2\right) \widehat{q}%
_{1,m-1}+\sum_{i=2}^{m-1}\widehat{q}_{i,m-1}}{q_{1}q_{2}...q_{m-1}}} \\
&=&2^{\frac{\left( m-1\right) \widehat{q}_{1,m}+\sum_{i=2}^{m}\widehat{q}%
_{i,m}}{q_{1}q_{2}...q_{m}}}.
\end{eqnarray*}%
From \cite{diniz} we know that $\left\Vert T_{m}\right\Vert =2^{m-1}$, and
therefore
\begin{equation*}
C_{m,\mathbf{q}}^{\mathbb{R}}\geq \frac{2^{\frac{\left( m-1\right) \widehat{q%
}_{1,m}+\sum_{i=2}^{m}\widehat{q}_{i,m}}{q_{1}q_{2}...q_{m}}}}{2^{m-1}}\geq
2^{\frac{\left( m-1\right) \widehat{q}_{1,m}+\left( \sum_{i=2}^{m}\widehat{q}%
_{i,m}\right) -\left( m-1\right) q_{1}q_{2}...q_{m}}{q_{1}q_{2}...q_{m}}},
\end{equation*}%
and the proof is done.
\end{proof}

\bigskip

Note that when $q_{1}=...=q_{m}=\frac{2m}{m+1}$, it follows that for all $%
i\in \left\{ 1,...,m\right\} $,

\begin{equation*}
\widehat{q}_{i,m}=\left( \frac{2m}{m+1}\right) ^{m-1}
\end{equation*}%
and thus, when $q_{1}=...=q_{m}=\frac{2m}{m+1},$ the inequality (\ref{lower}%
) recovers
\begin{eqnarray}
C_{m,\mathbf{q}}^{\mathbb{R}} &\geq &2^{\frac{\left( m-1\right) \left( \frac{%
2m}{m+1}\right) ^{m-1}+\left( \sum_{i=2}^{m}\left( \frac{2m}{m+1}\right)
^{m-1}\right) -\left( m-1\right) \left( \frac{2m}{m+1}\right) ^{m}}{\left(
\frac{2m}{m+1}\right) ^{m}}}  \label{gggg6} \\
&=&2^{\frac{\left( m-1\right) \left( \frac{2m}{m+1}\right) ^{m-1}+\left(
m-1\right) \left( \frac{2m}{m+1}\right) ^{m-1}-\left( m-1\right) \left(
\frac{2m}{m+1}\right) ^{m}}{\left( \frac{2m}{m+1}\right) ^{m}}}  \notag \\
&=&2^{\frac{m-1}{m}}  \notag
\end{eqnarray}%
that is precisely the lower estimate from \cite{diniz}. Besides, for \ $%
\mathbf{q}=(\alpha ,\beta _{m},...,\beta _{m})$, we have
\begin{eqnarray}
C_{m,\mathbf{q}}^{\mathbb{R}} &\geq &2^{\frac{\left( m-1\right) \beta
_{m}^{m-1}+\left( \sum_{i=2}^{m}\alpha \beta _{m}^{m-2}\right) -\left(
m-1\right) \alpha \beta _{m}^{m-1}}{\alpha \beta ^{m-1}}}  \label{gggg7} \\
&=&2^{\frac{\left( \alpha \beta _{m}^{m-2}+\beta _{m}^{m-1}\right) \left(
m-1\right) -\left( m-1\right) \alpha \beta _{m}^{m-1}}{\alpha \beta
_{m}^{m-1}}}  \notag \\
&=&2^{\frac{1}{\alpha \beta _{m}}\left( m-1\right) \left( \alpha +\beta
_{m}-\alpha \beta _{m}\right) }  \notag \\
&=&2^{\frac{1}{2\alpha }\left( 2m+3\alpha -m\alpha -4\right) },  \notag
\end{eqnarray}%
which is the estimate from \ \cite{natal}, if we use $\beta _{m}=\frac{%
2\alpha m-2\alpha }{\alpha m-2+\alpha }$. In particular for \ $\mathbf{q}%
=(1,2,...,2)$, we have

\begin{equation}
C_{m,\mathbf{q}}^{\mathbb{R}}\geq 2^{\frac{m-1}{2}},  \label{gggg8}
\end{equation}%
recovering the main result of \cite{natal}.

\section{Second main result}

\bigskip Our second main result shows that the same estimate obtained in
\cite{natal} also holds for exponents of the type $\left(
2,1,2,2,...,2\right) .$

\bigskip For $m=2$ let us define the bilinear operator $L_{2}$ as $T_{2}$
from the previous section. For $m\geq 3,$ consider

\begin{equation*}
\begin{tabular}{llll}
$L_{m}:$ & $\ell _{\infty }^{2^{m-1}}\times \cdots \times \ell _{\infty
}^{2^{m-1}}$ & $\rightarrow $ & $\mathbb{R}$ \\
& $\left( x^{\left( 1\right) },...,x^{\left( m\right) }\right) $ & $\mapsto $
& $\left( x_{1}^{\left( 1\right) }+x_{2}^{\left( 1\right) }\right)
T_{m-1}\left( x^{\left( 2\right) },...,x^{\left( m\right) }\right) $ \\
&  &  & $+\left( x_{1}^{\left( 1\right) }-x_{2}^{\left( 1\right) }\right)
T_{m-1}\left( \widehat{x}^{\left( 2\right) },\widehat{x}^{\left( 3\right)
},...,\widehat{x}^{\left( m\right) }\right) ,$%
\end{tabular}%
\end{equation*}%
where
\begin{eqnarray*}
\widehat{x}^{\left( 2\right) } &=&B^{2^{m-2}}\left( x^{\left( 2\right)
}\right) , \\
\widehat{x}^{\left( 3\right) } &=&B^{2^{m-2}}\left( x^{\left( 3\right)
}\right) , \\
\widehat{x}^{\left( 4\right) } &=&B^{2^{m-3}}\left( x^{\left( 4\right)
}\right) , \\
\widehat{x}^{\left( 5\right) } &=&B^{2^{m-4}}\left( x^{\left( 5\right)
}\right) , \\
&&\vdots \\
\widehat{x}^{\left( m-1\right) } &=&B^{2^{2}}\left( x^{\left( m-1\right)
}\right) , \\
\widehat{x}^{\left( m\right) } &=&B^{2}\left( x^{\left( m\right) }\right) ,
\end{eqnarray*}%
and $B$ is the backward shift operator in $\ell _{\infty }^{2^{m-1}}$. Using
the previous theorem we get the following:

\begin{theorem}
\label{outro_opera} If\textit{\ }$m\geq 2$ is a positive integer, \textit{and%
} $\mathbf{q}:=(q_{1},...,q_{m})\in \left[ 1,2\right] ^{m}$ is \textit{such
that}
\begin{equation*}
\frac{1}{q_{1}}+\cdots +\frac{1}{q_{m}}=\frac{m+1}{2},
\end{equation*}%
t\textit{hen there exists a constant}
\begin{equation*}
C_{m,\mathbf{q}}^{\mathbb{R}}\geq 2^{\frac{\left( m-1\right) \widehat{q}%
_{2,m}+\left( \sum_{\underset{i\not=2}{i=1}}^{m}\widehat{q}_{i,m}\right)
-\left( m-1\right) q_{1}q_{2}...q_{m}}{q_{1}q_{2}...q_{m}}}
\end{equation*}%
\textit{such that}
\begin{equation*}
\left( \sum_{j_{1}=1}^{n}\left( \sum_{j_{2}=1}^{n}\left( \cdots \left(
\sum_{j_{m}=1}^{n}\left\vert T(e_{j_{1}},...,e_{j_{m}})\right\vert
^{q_{m}}\right) ^{\frac{q_{m-1}}{q_{m}}}\cdots \right) ^{\frac{q_{2}}{q_{3}}%
}\right) ^{\frac{q_{1}}{q_{2}}}\right) ^{\frac{1}{q_{1}}}\leq C_{m,\mathbf{q}%
}^{\mathbb{R}}\left\Vert T\right\Vert
\end{equation*}%
\textit{for all continuous }$m$\textit{--linear forms }$T:c_{0}\times \cdots
\times c_{0}\rightarrow \mathbb{R}$\textit{\ and all positive integers }$n.$
\end{theorem}

\begin{proof}
For $m=2$ the result is encompassed by Theorem \ref{general}. Recall that
for $m\geq 3$, we have

\begin{equation*}
\begin{tabular}{llll}
$L_{m}:$ & $\ell _{\infty }^{2^{m-1}}\times \cdots \times \ell _{\infty
}^{2^{m-1}}$ & $\rightarrow $ & $\mathbb{R}$ \\
& $\left( x^{\left( 1\right) },...,x^{\left( m\right) }\right) $ & $\mapsto $
& $\left( x_{1}^{\left( 1\right) }+x_{2}^{\left( 1\right) }\right)
T_{m-1}\left( x^{\left( 2\right) },...,x^{\left( m\right) }\right) $ \\
&  &  & $+\left( x_{1}^{\left( 1\right) }-x_{2}^{\left( 1\right) }\right)
T_{m-1}\left( B^{2^{m-2}}\left( x^{\left( 2\right) }\right)
,B^{2^{m-2}}\left( x^{\left( 3\right) }\right) ,...,B^{2}\left( x^{\left(
m\right) }\right) \right) ,$%
\end{tabular}%
\end{equation*}%
where
\begin{equation*}
x^{\left( k\right) }=\left( x_{j}^{\left( k\right) }\right)
_{j=1}^{2^{m-1}}\in \ell _{\infty }^{2^{m-1}},
\end{equation*}%
$1\leq k\leq m$, and $B$ is the backward shift operator in $\ell _{\infty
}^{2^{m-1}}.$ We can again realize that we could consider the domain of $%
L_{m}$ as $\ell _{\infty }^{2}\times \ell _{\infty }^{2^{m-1}}\times \ell
_{\infty }^{2^{m-1}}\times \ell _{\infty }^{2^{m-2}}\times \ell _{\infty
}^{2^{m-3}}\times \cdots \times \ell _{\infty }^{2^{2}}$. Note also that%
\begin{equation*}
\left\Vert L_{m}\right\Vert =\left\Vert T_{m}\right\Vert =2^{m-1}.
\end{equation*}

Let us see that for $m\geq 2$ and $1\leq i\leq m$ we have%
\begin{equation*}
\left( \sum_{j_{1}=1}^{2^{m-1}}\left( \sum_{j_{2}=1}^{2^{m-1}}\left( \cdots
\left( \sum_{j_{m}=1}^{2^{m-1}}\left\vert
L_{m}(e_{j_{1}},...,e_{j_{m}})\right\vert ^{q_{m}}\right) ^{\frac{q_{m-1}}{%
q_{m}}}\cdots \right) ^{\frac{q_{2}}{q_{3}}}\right) ^{\frac{q_{1}}{q_{2}}%
}\right) ^{\frac{1}{q_{1}}}=2^{\frac{\left( m-1\right) \widehat{q}%
_{2,m}+\sum_{\underset{i\not=2}{i=1}}^{m}\widehat{q}_{i,m}}{%
q_{1}q_{2}...q_{m}}}.
\end{equation*}

In fact, note that \
\begin{eqnarray*}
&&\left( \sum_{j_{1}=1}^{2^{m-1}}\left( \sum_{j_{2}=1}^{2^{m-1}}\left(
\cdots \left( \sum_{j_{m}=1}^{2^{m-1}}\left\vert
L_{m}(e_{j_{1}},...,e_{j_{m}})\right\vert ^{q_{m}}\right) ^{\frac{q_{m-1}}{%
q_{m}}}\cdots \right) ^{\frac{q_{2}}{q_{3}}}\right) ^{\frac{q_{1}}{q_{2}}%
}\right) ^{\frac{1}{q_{1}}} \\
&=&\left(
\begin{array}{c}
\left( \sum_{j_{2}=1}^{2^{m-1}}\left( \cdots \left(
\sum_{j_{m}=1}^{2^{m-1}}\left\vert
L_{m}(e_{1},e_{j_{2}},...,e_{j_{m}})\right\vert ^{q_{m}}\right) ^{\frac{%
q_{m-1}}{q_{m}}}\cdots \right) ^{\frac{q_{2}}{q_{3}}}\right) ^{\frac{q_{1}}{%
q_{2}}} \\
+\left( \sum_{j_{2}=1}^{2^{m-1}}\left( \cdots \left(
\sum_{j_{m}=1}^{2^{m-1}}\left\vert
L_{m}(e_{2},e_{j_{2}},...,e_{j_{m}})\right\vert ^{q_{m}}\right) ^{\frac{%
q_{m-1}}{q_{m}}}\cdots \right) ^{\frac{q_{2}}{q_{3}}}\right) ^{\frac{q_{1}}{%
q_{2}}}%
\end{array}%
\right) ^{\frac{1}{q_{1}}}
\end{eqnarray*}%
and also that%
\begin{eqnarray*}
&&\left(
\begin{array}{c}
\left( \sum_{j_{2}=1}^{2^{m-1}}\left( \cdots \left(
\sum_{j_{m}=1}^{2^{2}}\left\vert
\begin{array}{c}
T_{m-1}\left( e_{j_{2}},...,e_{j_{m}}\right) \\
+T_{m-1}\left( B^{2^{m-2}}\left( e_{j_{2}}\right) ,B^{2^{m-2}}\left(
e_{j_{3}}\right) ,...,B^{2}\left( e_{j_{m}}\right) \right)%
\end{array}%
\right\vert ^{q_{m}}\right) ^{\frac{q_{m-1}}{q_{m}}}\cdots \right) ^{\frac{%
q_{2}}{q_{3}}}\right) ^{\frac{q_{1}}{q_{2}}} \\
+\left( \sum_{j_{2}=1}^{2^{m-1}}\left( \cdots \left(
\sum_{j_{m}=1}^{2^{2}}\left\vert
\begin{array}{c}
T_{m-1}\left( e_{j_{2}},...,e_{j_{m}}\right) \\
-T_{m-1}\left( B^{2^{m-2}}\left( e_{j_{2}}\right) ,B^{2^{m-2}}\left(
e_{j_{3}}\right) ,...,B^{2}\left( e_{j_{m}}\right) \right)%
\end{array}%
\right\vert ^{q_{m}}\right) ^{\frac{q_{m-1}}{q_{m}}}\cdots \right) ^{\frac{%
q_{2}}{q_{3}}}\right) ^{\frac{q_{1}}{q_{2}}}%
\end{array}%
\right) ^{\frac{1}{q_{1}}} \\
&=&2^{\frac{1}{q_{1}}}\left( \sum_{j_{2}=1}^{2^{m-1}}\left( \cdots
\sum_{j_{m-1}=1}^{2^{3}}\left(
\begin{array}{c}
\left\vert T_{m-1}\left( e_{j_{2}},...,e_{1}\right) \right\vert ^{q_{m}} \\
+\left\vert T_{m-1}\left( e_{j_{2}},...,e_{2}\right) \right\vert ^{q_{m}} \\
+\left\vert T_{m-1}\left( B^{2^{m-2}}\left( e_{j_{2}}\right)
,B^{2^{m-2}}\left( e_{j_{3}}\right) ,...,B^{2}\left( e_{3}\right) \right)
\right\vert ^{q_{m}} \\
+\left\vert T_{m-1}\left( B^{2^{m-2}}\left( e_{j_{2}}\right)
,B^{2^{m-2}}\left( e_{j_{3}}\right) ,...,B^{2}\left( e_{4}\right) \right)
\right\vert ^{q_{m}}%
\end{array}%
\right) ^{\frac{q_{m-1}}{q_{m}}}\cdots \right) ^{\frac{q_{2}}{q_{3}}}\right)
^{\frac{1}{q_{2}}}.
\end{eqnarray*}

Moreover, observe that%
\begin{eqnarray*}
&&T_{m-1}\left( B^{2^{m-2}}\left( e_{j_{2}}\right) ,B^{2^{m-2}}\left(
e_{j_{3}}\right) ,B^{2^{m-3}}\left( e_{j_{4}}\right) ,...,B^{4}\left(
e_{j_{m-1}}\right) ,B^{2}\left( e_{3}\right) \right) \\
&=&T_{m-1}\left(
e_{j_{2}-2^{m-2}},e_{j_{3}-2^{m-2}},e_{j_{4}-2^{m-3}},...,e_{j_{m-1}-4},e_{1}\right)
\end{eqnarray*}%
and%
\begin{eqnarray*}
&&T_{m-1}\left( B^{2^{m-2}}\left( e_{j_{2}}\right) ,B^{2^{m-2}}\left(
e_{j_{3}}\right) ,B^{2^{m-3}}\left( e_{j_{4}}\right) ,...,B^{4}\left(
e_{j_{m-1}}\right) ,B^{2}\left( e_{4}\right) \right) \\
&=&T_{m-1}\left(
e_{j_{2}-2^{m-2}},e_{j_{3}-2^{m-2}},e_{j_{4}-2^{m-3}},...,e_{j_{m-1}-4},e_{2}\right) .
\end{eqnarray*}%
Fir the sake of simplicity, let us define
\begin{equation*}
A_{1}:=\left\vert T_{m-1}\left( e_{j_{2}},...,e_{1}\right) \right\vert
^{q_{m}}+\left\vert T_{m-1}\left( e_{j_{2}},...,e_{2}\right) \right\vert
^{q_{m}}
\end{equation*}%
and%
\begin{eqnarray*}
A_{2} &:&=\left\vert T_{m-1}\left(
e_{j_{2}-2^{m-2}},e_{j_{3}-2^{m-2}},e_{j_{4}-2^{m-3}},...,e_{j_{m-1}-4},e_{1}\right) \right\vert ^{q_{m}}
\\
&&+\left\vert T_{m-1}\left(
e_{j_{2}-2^{m-2}},e_{j_{3}-2^{m-2}},e_{j_{4}-2^{m-3}},...,e_{j_{m-1}-4},e_{2}\right) \right\vert ^{q_{m}}.
\end{eqnarray*}

\bigskip

According to the definition of $T_{m-1}$ we know that $A_{1}$ is non null
only when $j_{m-1}\in \left\{ 1,2,3,4\right\} ,$ and $j_{m-2}\in \left\{
1,2,...,2^{3}\right\} $, ..., $j_{3},j_{2}\in \left\{
1,2,...,2^{m-2}\right\} $. Analogously, $A_{2}$ is non null only when $%
j_{m-1}\in \left\{ 5,6,7,8\right\} ,$ and $j_{m-2}\in \left\{
9,2,...,2^{4}\right\} $, ..., $j_{3},j_{2}\in \left\{
2^{m-2}+1,...,2^{m-1}\right\} $.

Therefore%
\begin{eqnarray*}
&&\sum_{j_{2}=1}^{2^{m-1}}\left( \cdots \sum_{j_{m-1}=1}^{2^{3}}\left(
\begin{array}{c}
\left\vert T_{m-1}\left( e_{j_{2}},...,e_{1}\right) \right\vert ^{q_{m}} \\
+\left\vert T_{m-1}\left( e_{j_{2}},...,e_{2}\right) \right\vert ^{q_{m}} \\
+\left\vert T_{m-1}\left(
e_{j_{2}-2^{m-2}},e_{j_{3}-2^{m-2}},e_{j_{4}-2^{m-3}},...,e_{j_{m-1}-4},e_{1}\right) \right\vert ^{q_{m}}
\\
+\left\vert T_{m-1}\left(
e_{j_{2}-2^{m-2}},e_{j_{3}-2^{m-2}},e_{j_{4}-2^{m-3}},...,e_{j_{m-1}-4},e_{2}\right) \right\vert ^{q_{m}}%
\end{array}%
\right) ^{\frac{q_{m-1}}{q_{m}}}\cdots \right) ^{\frac{q_{2}}{q_{3}}} \\
&=&\sum_{j_{2}=1}^{2^{m-2}}\left( \cdots \sum_{j_{m-1}=1}^{2^{2}}\left(
\begin{array}{c}
\left\vert T_{m-1}\left( e_{j_{2}},...,e_{1}\right) \right\vert ^{q_{m}} \\
+\left\vert T_{m-1}\left( e_{j_{2}},...,e_{2}\right) \right\vert ^{q_{m}}%
\end{array}%
\right) ^{\frac{q_{m-1}}{q_{m}}}\cdots \right) ^{\frac{q_{2}}{q_{3}}} \\
&&+\sum_{j_{2}=2^{m-2}+1}^{2^{m-1}}\left( \cdots
\sum_{j_{m-1}=2^{2}+1}^{2^{3}}\left(
\begin{array}{c}
\left\vert T_{m-1}\left(
e_{j_{2}-2^{m-2}},e_{j_{3}-2^{m-2}},e_{j_{4}-2^{m-3}},...,e_{j_{m-1}-4},e_{1}\right) \right\vert ^{q_{m}}
\\
+\left\vert T_{m-1}\left(
e_{j_{2}-2^{m-2}},e_{j_{3}-2^{m-2}},e_{j_{4}-2^{m-3}},...,e_{j_{m-1}-4},e_{2}\right) \right\vert ^{q_{m}}%
\end{array}%
\right) ^{\frac{q_{m-1}}{q_{m}}}\cdots \right) ^{\frac{q_{2}}{q_{3}}}
\end{eqnarray*}%
and re-writing the indices of the last sum we have%
\begin{eqnarray*}
&=&\sum_{j_{2}=1}^{2^{m-2}}\left( \cdots \sum_{j_{m-1}=1}^{2^{2}}\left(
\begin{array}{c}
\left\vert T_{m-1}\left( e_{j_{2}},...,e_{1}\right) \right\vert ^{q_{m}} \\
+\left\vert T_{m-1}\left( e_{j_{2}},...,e_{2}\right) \right\vert ^{q_{m}}%
\end{array}%
\right) ^{\frac{q_{m-1}}{q_{m}}}\cdots \right) ^{\frac{q_{2}}{q_{3}}} \\
&&+\sum_{j_{2}=2^{m-2}+1}^{2^{m-1}}\left( \cdots
\sum_{j_{m-1}=2^{2}+1}^{2^{3}}\left(
\begin{array}{c}
\left\vert T_{m-1}\left(
e_{j_{2}-2^{m-2}},e_{j_{3}-2^{m-2}},e_{j_{4}-2^{m-3}},...,e_{j_{m-1}-4},e_{1}\right) \right\vert ^{q_{m}}
\\
+\left\vert T_{m-1}\left(
e_{j_{2}-2^{m-2}},e_{j_{3}-2^{m-2}},e_{j_{4}-2^{m-3}},...,e_{j_{m-1}-4},e_{2}\right) \right\vert ^{q_{m}}%
\end{array}%
\right) ^{\frac{q_{m-1}}{q_{m}}}\cdots \right) ^{\frac{q_{2}}{q_{3}}} \\
&=&\sum_{j_{2}=1}^{2^{m-2}}\left( \cdots \sum_{j_{m-1}=1}^{2^{2}}\left(
\begin{array}{c}
\left\vert T_{m-1}\left( e_{j_{2}},...,e_{1}\right) \right\vert ^{q_{m}} \\
+\left\vert T_{m-1}\left( e_{j_{2}},...,e_{2}\right) \right\vert ^{q_{m}}%
\end{array}%
\right) ^{\frac{q_{m-1}}{q_{m}}}\cdots \right) ^{\frac{q_{2}}{q_{3}}} \\
&&+\sum_{j_{2}=1}^{2^{m-2}}\left( \cdots \sum_{j_{m-1}=1}^{2^{2}}\left(
\begin{array}{c}
\left\vert T_{m-1}\left( e_{j_{2}},...,e_{1}\right) \right\vert ^{q_{m}} \\
+\left\vert T_{m-1}\left( e_{j_{2}},...,e_{2}\right) \right\vert ^{q_{m}}%
\end{array}%
\right) ^{\frac{q_{m-1}}{q_{m}}}\cdots \right) ^{\frac{q_{2}}{q_{3}}}
\end{eqnarray*}%
\begin{eqnarray*}
&=&2\left( \sum_{j_{2}=1}^{2^{m-2}}\left( \cdots
\sum_{j_{m-1}=1}^{2^{2}}\left(
\begin{array}{c}
\left\vert T_{m-1}\left( e_{j_{2}},...,e_{1}\right) \right\vert ^{q_{m}} \\
+\left\vert T_{m-1}\left( e_{j_{2}},...,e_{2}\right) \right\vert ^{q_{m}}%
\end{array}%
\right) ^{\frac{q_{m-1}}{q_{m}}}\cdots \right) ^{\frac{q_{2}}{q_{3}}}\right)
\\
&=&2\left( \sum_{j_{2}=1}^{2^{m-2}}\left( \cdots
\sum_{j_{m-1}=1}^{2^{2}}\left( \sum_{j_{m}=1}^{2}\left\vert T_{m-1}\left(
e_{j_{2}},...,e_{j_{m}}\right) \right\vert ^{q_{m}}\right) ^{\frac{q_{m-1}}{%
q_{m}}}\cdots \right) ^{\frac{q_{2}}{q_{3}}}\right) \\
&=&2\left( 2^{\frac{\left( m-2\right) \frac{\hat{q}_{2,m}}{q_{1}}%
+\sum_{i=3}^{m}\frac{\hat{q}_{i,m}}{q_{1}}}{q_{2}...q_{m}}}\right) ^{q_{2}}.
\end{eqnarray*}%
Hence
\begin{eqnarray*}
&&2^{\frac{1}{q_{1}}}\left( \sum_{j_{2}=1}^{2^{m-1}}\left( \cdots
\sum_{j_{m-1}=1}^{2^{3}}\left(
\begin{array}{c}
\left\vert T_{m-1}\left( e_{j_{2}},...,e_{1}\right) \right\vert ^{q_{m}} \\
+\left\vert T_{m-1}\left( e_{j_{2}},...,e_{2}\right) \right\vert ^{q_{m}} \\
+\left\vert T_{m-1}\left( B^{2^{m-2}}\left( e_{j_{2}}\right)
,B^{2^{m-2}}\left( e_{j_{3}}\right) ,...,B^{2}\left( e_{3}\right) \right)
\right\vert ^{q_{m}} \\
+\left\vert T_{m-1}\left( B^{2^{m-2}}\left( e_{j_{2}}\right)
,B^{2^{m-2}}\left( e_{j_{3}}\right) ,...,B^{2}\left( e_{4}\right) \right)
\right\vert ^{q_{m}}%
\end{array}%
\right) ^{\frac{q_{m-1}}{q_{m}}}\cdots \right) ^{\frac{q_{2}}{q_{3}}}\right)
^{\frac{1}{q_{2}}} \\
&=&2^{\frac{1}{q_{1}}}\left( 2\left( 2^{\frac{\left( m-2\right) \frac{\hat{q}%
_{2,m}}{q_{1}}+\sum_{i=3}^{m}\frac{\hat{q}_{i,m}}{q_{1}}}{q_{2}...q_{m}}%
}\right) ^{q_{2}}\right) ^{\frac{1}{q_{2}}} \\
&=&2^{\frac{1}{q_{1}}}2^{\frac{1}{q_{2}}}2^{\frac{\left( m-2\right) \frac{%
\hat{q}_{2,m}}{q_{1}}+\sum_{i=3}^{m}\frac{\hat{q}_{i,m}}{q_{1}}}{%
q_{2}...q_{m}}}.
\end{eqnarray*}

Finally, since $\left\Vert L_{m}\right\Vert =2^{m-1}$, we have
\begin{equation*}
C_{m,\mathbf{q}}^{\mathbb{R}}\geq \frac{2^{\frac{\left( m-1\right) \hat{q}%
_{2,m}+\left( \sum_{\underset{i\not=2}{i=1}}^{m}\widehat{q}_{i,m}\right) }{%
q_{1}q_{2}...q_{m}}}}{2^{m-1}}=2^{\left( \left( m-1\right) \hat{q}%
_{2,m}+\left( \sum_{\underset{i\not=2}{i=1}}^{m}\widehat{q}_{i,m}\right)
-\left( m-1\right) q_{1}q_{2}...q_{m}\right) \left(
q_{1}q_{2}...q_{m}\right) ^{-1}}.
\end{equation*}
\end{proof}

\begin{corollary}
The optimal constants of the mixed $\left( \ell _{1},\ell _{2}\right) $%
-Littlewood-type inequality for $\mathbf{q=}\left( 2,1,2,...,2\right) $ is $%
C_{m,\left( 2,1,2,...,2\right) }^{\mathbb{R}}=2^{\frac{m-1}{2}}.$
\end{corollary}

\begin{proof}
We have
\begin{eqnarray*}
C_{m,\mathbf{q}}^{\mathbb{R}} &\geq &2^{\frac{\left( m-1\right) \widehat{q}%
_{2,m}+\left( \sum_{\underset{i\not=2}{i=1}}^{m}\widehat{q}_{i,m}\right)
-\left( m-1\right) q_{1}q_{2}...q_{m}}{q_{1}q_{2}...q_{m}}} \\
&=&2^{\frac{\left( m-1\right) 2^{m-1}+\left( m-1\right) \left(
2^{m-2}\right) -\left( m-1\right) 2^{m-1}}{2^{m-1}}} \\
&=&2^{\frac{m-1}{2}}.
\end{eqnarray*}%
On the other hand, since $C_{m,\mathbf{q}}^{\mathbb{R}}\leq 2^{\frac{m-1}{2}%
} $ (see \cite{abps}), we conclude that
\begin{equation*}
C_{m,\left( 2,1,2,...,2\right) }^{\mathbb{R}}=2^{\frac{m-1}{2}}.  \label{212}
\end{equation*}
\end{proof}

\section{Some remarks on the upper estimates of the general
Bohnenblust--Hille inequality}

In this section we extend some recent results providing upper estimates for
the generalized Bohnenblust--Hille inequality. These results will be used in
the final section.\ We begin by recalling two results:

\begin{lemma}[{\protect\cite[Lemma 2.1]{arapel}}]
\label{arapel1}Let $\left( q_{1},...,q_{m}\right) \in \lbrack 1,2]^{m}$ such
that $\frac{1}{q_{1}}+...+\frac{1}{q_{m}}=\frac{m+1}{2}$. If $q_{i}\geq
\frac{2m-2}{m}$ for some index $i$ and $q_{k}=q_{l}$ for all $k\not=i$ and $%
l\not=i$, then
\begin{equation*}
B_{m,\left( q_{1},...,q_{m}\right) }^{\mathbb{K}}\leq
\prod\limits_{j=2}^{m}A_{\frac{2j-2}{j}}^{-1}
\end{equation*}%
where $A_{\frac{2j-2}{j}}$ are the respective constants of the Khinchine
inequality.
\end{lemma}

\begin{theorem}[{\protect\cite[Theorem 2.3]{arapel}}]
\label{arapel2} If\textit{\ }$m\geq 2$ is a positive integer, \textit{and} $%
\mathbf{q}:=(q_{1},...,q_{m})\in \left[ 1,2\right] ^{m}$ are \textit{such
that}
\begin{equation*}
\frac{1}{q_{1}}+\cdots +\frac{1}{q_{m}}=\frac{m+1}{2},
\end{equation*}%
and
\begin{equation*}
\max q_{i}<\frac{2m^{2}-4m+2}{m^{2}-m-1}
\end{equation*}%
t\textit{hen}%
\begin{equation*}
C_{m,\mathbf{q}}^{\mathbb{K}}\leq \prod_{j=2}^{m}A_{\frac{2j-2}{j}}^{-1}.
\end{equation*}
\end{theorem}

\bigskip Combining these two results and using an interpolative approach
(see \cite[Proposition 2.1]{abps} and \cite[Lemma 2.1]{anss}) we can prove
the following slightly more general result.

\begin{proposition}
\label{up}Let\textit{\ }$m\geq 2$ and $N~$be positive integers \textit{and} $%
~\mathbf{q,q}\left( 1\right) ,...,\mathbf{q}\left( N\right) \in \left[ 1,2%
\right] ^{m}$ be such that $\left( \frac{1}{q_{1}},...,\frac{1}{q_{m}}%
\right) $ belong to the convex hull of $\left( \frac{1}{q_{1}\left( k\right)
},...,\frac{1}{q_{m}\left( k\right) }\right) $, $k=1,...,N,$ where
\begin{equation*}
\frac{1}{q_{1}\left( k\right) }+\cdots +\frac{1}{q_{m}\left( k\right) }=%
\frac{m+1}{2},\text{ for all }k=1,...,N.
\end{equation*}%
If, for each $k=1,...,N$,
\begin{equation*}
\max_{i}q_{i}\left( k\right) <\frac{2m^{2}-4m+2}{m^{2}-m-1}
\end{equation*}%
or%
\begin{equation*}
q_{i}\left( k\right) \geq \frac{2m-2}{m}
\end{equation*}%
for some index $i$ and $q_{j}\left( k\right) =q_{l}\left( k\right) $ for all
$j\not=i$ and $l\not=i$, t\textit{hen}%
\begin{equation*}
C_{m,\mathbf{q}}^{\mathbb{K}}\leq \prod_{t=2}^{m}A_{\frac{2t-2}{t}}^{-1}.
\end{equation*}
\end{proposition}

\begin{proof}
Let us suppose that for each $q_{i}\left( k\right) $, with $k=1,...,N$,
\begin{equation*}
\max_{i}q_{i}\left( k\right) <\frac{2m^{2}-4m+2}{m^{2}-m-1}
\end{equation*}%
or%
\begin{equation*}
q_{i}\left( k\right) \geq \frac{2m-2}{m}
\end{equation*}%
for some index $i$ and $q_{j}\left( k\right) =q_{l}\left( k\right) $ for all
$j\not=i$ and $l\not=i$. No matter what situation happens, for each $k\in
\{1,...,N\}$, we have from Lemma \ref{arapel1}\ or from Theorem \ref{arapel2}
that
\begin{equation*}
C_{m,\mathbf{q}\left( k\right) }^{\mathbb{K}}\leq \prod_{j=2}^{m}A_{\frac{%
2j-2}{j}}^{-1}.
\end{equation*}%
Now, suppose that $\mathbf{q=}\left( \frac{1}{q_{1}},...,\frac{1}{q_{m}}%
\right) ~$belongs to the convex hull of $\left( \frac{1}{q_{1}\left(
k\right) },...,\frac{1}{q_{m}\left( k\right) }\right) ,k=1,...,N,$ i. e.
\begin{equation*}
\frac{1}{q_{i}}=\frac{\theta _{1}}{q_{i}\left( 1\right) }+...+\frac{\theta
_{N}}{q_{i}\left( N\right) }
\end{equation*}%
with $\sum_{k=1}^{N}\theta _{k}=1$ and $\theta _{k}\in \left[ 0,1\right] $
for all $k.$ So, by the interpolation procedure from \cite[Proposition 2.1]%
{abps}, we have
\begin{equation*}
\left\Vert T(e_{j_{1}},...,e_{j_{m}})\right\Vert _{\mathbf{q}}\leq
\prod\limits_{k=1}^{N}\left\Vert T(e_{j_{1}},...,e_{j_{m}})\right\Vert _{%
\mathbf{q}\left( k\right) }^{\theta _{k}}.
\end{equation*}%
From Lemma $\ref{arapel1}\ $and/or from Theorem $\ref{arapel2}$ we have
\begin{eqnarray*}
\prod\limits_{k=1}^{N}\left\Vert T(e_{j_{1}},...,e_{j_{m}})\right\Vert _{%
\mathbf{q}\left( k\right) }^{\theta _{k}} &\leq
&\prod\limits_{k=1}^{N}\left( \prod_{t=2}^{m}A_{\frac{2t-2}{t}%
}^{-1}\left\Vert T\right\Vert \right) ^{\theta _{k}} \\
&=&\prod_{t=2}^{m}A_{\frac{2t-2}{t}}^{-1}\left\Vert T\right\Vert
\end{eqnarray*}%
and thus
\begin{eqnarray*}
\left( \sum_{j_{1}=1}^{n}\left( \sum_{j_{2}=1}^{n}\left( \cdots \left(
\sum_{j_{m}=1}^{n}\left\vert T(e_{j_{1}},...,e_{j_{m}})\right\vert
^{q_{m}}\right) ^{\frac{q_{m-1}}{q_{m}}}\cdots \right) ^{\frac{q_{2}}{q_{3}}%
}\right) ^{\frac{q_{1}}{q_{2}}}\right) ^{\frac{1}{q_{1}}} &=&\left\Vert
T(e_{j_{1}},...,e_{j_{m}})\right\Vert _{\mathbf{q}} \\
&\leq &\prod_{t=2}^{m}A_{\frac{2t-2}{t}}^{-1}\left\Vert T\right\Vert .
\end{eqnarray*}
\end{proof}

\bigskip

\begin{example}
Note that the above result encompasses cases not covered by the previous
results, and we still have%
\begin{equation*}
C_{m,\mathbf{q}}^{\mathbb{K}}\leq \prod_{t=2}^{m}A_{\frac{2t-2}{t}}^{-1}.
\end{equation*}%
For instance, suppose that $m=N=3$ and $\mathbf{q}\left( 1\right) =\left( 2,%
\frac{4}{3},\frac{4}{3}\right) $, $\mathbf{q}\left( 2\right) =\left( \frac{4%
}{3},2,\frac{4}{3}\right) \,$, $\mathbf{q}\left( 3\right) =\left( \frac{4}{3}%
,\frac{4}{3},2\right) $. So, from the Proposition \ref{up}, for
\begin{equation*}
\mathbf{q}=\left( \frac{4}{3-\theta _{1}},\frac{4}{3-\theta _{2}},\frac{4}{%
3-\theta _{3}}\right) ;~\theta _{1},\theta _{2},\theta _{3}\in \lbrack 0,1]%
\text{ and }\theta _{1}+\theta _{2}+\theta _{3}=1
\end{equation*}%
we have
\begin{equation*}
C_{3,\mathbf{q}}^{\mathbb{K}}\leq \prod_{t=2}^{3}A_{\frac{2t-2}{t}}^{-1}.
\end{equation*}%
Considering $\left( \theta _{1},\theta _{2},\theta _{3}\right) =\left( \frac{%
7}{10},\frac{1}{10},\frac{2}{10}\right) $ we have
\begin{equation*}
\mathbf{q}=\left( \frac{40}{23},\frac{40}{29},\frac{40}{28}\right)
\end{equation*}%
and, of course, $\mathbf{q}$ does not satisfy the hypotheses of Lemma \ref%
{arapel1}, and since%
\begin{equation*}
\frac{40}{23}>\frac{2\left( 3\right) ^{2}-4\left( 3\right) +2}{\left(
3\right) ^{2}-\left( 3\right) -1}=\allowbreak \frac{8}{5},
\end{equation*}%
$\mathbf{q}$ also does not satisfy the hypotheses of Theorem \ref{arapel2}.
\end{example}

\bigskip\

\section{Application: Sharp estimates for the general Bohnenblust--Hille
inequality for $3$-linear forms}

\bigskip In this final section we use the results of the previous sections
to obtain sharp estimates for the general Bohnenblust--Hille inequality for $%
3$-linear forms.

\begin{proposition}
\label{110}\bigskip \bigskip Let $\tau ,\theta \in \left[ 0,1\right] ^{2}$.
If
\begin{equation*}
\mathbf{q}=\left( \frac{4}{\theta +3},\frac{4}{2+\tau -\theta \tau },\frac{4%
}{3+\theta \tau -\theta -\tau }\right) ~~
\end{equation*}%
then, the optimal constant of the generalized Bohnenblust--Hille inequality
for real scalars is
\begin{equation*}
C_{3,\mathbf{q}}^{\mathbb{R}}=2^{\frac{\theta +3}{4}}.
\end{equation*}
\end{proposition}

\begin{proof}
When $m=3$ and $\mathbf{q}=(\alpha ,\beta ,\gamma )$, from Theorem \ref%
{general}, we have
\begin{equation}
C_{3,\mathbf{q}}^{\mathbb{R}}\geq 2^{\frac{2\beta \gamma +\alpha \beta
+\alpha \gamma -2\alpha \beta \gamma }{\alpha \beta \gamma }}.  \label{tri}
\end{equation}%
Let us first consider the case $\theta =0$. We can verify that the values of
$\left( \alpha ,\beta ,\gamma \right) \in \left[ 1,2\right] ^{3}$ with $%
\frac{1}{\alpha }+\frac{1}{\beta }+\frac{1}{\gamma }=2$ for which
\begin{equation*}
2^{\frac{2\beta \gamma +\alpha \beta +\alpha \gamma -2\alpha \beta \gamma }{%
\alpha \beta \gamma }}=2^{\frac{3}{4}}
\end{equation*}%
is precisely $\left( \frac{4}{3},\frac{4\gamma }{5\gamma -4},\gamma \right) $
with $\gamma \in \left[ \frac{4}{3},2\right] .$ Equivalently, $\left( \frac{4%
}{3},\frac{4}{2+\tau },\frac{4}{3-\tau }\right) $ with $\tau \in \left[ 0,1%
\right] $. Thus, using (\ref{tri}), if
\begin{equation*}
\mathbf{q}=\left( \frac{4}{3},\frac{4}{2+\tau },\frac{4}{3-\tau }\right) ~
\end{equation*}%
with $\tau \in \left[ 0,1\right] ,$ then
\begin{equation*}
C_{3,\mathbf{q}}^{\mathbb{R}}\geq 2^{\frac{3}{4}}.
\end{equation*}

On the other hand, from \cite[Lemma 2.1]{arapel} we know that for $\mathbf{q}%
\left( 1\right) =\left( \frac{4}{3},2,\frac{4}{3}\right) ,$ $\mathbf{q}%
\left( 2\right) =\left( \frac{4}{3},\frac{4}{3},2\right) $\textbf{, }we have
$C_{3,\mathbf{q}\left( 1\right) }^{\mathbb{R}}\leq 2^{\frac{3}{4}}$ and $%
C_{3,\mathbf{q}\left( 2\right) }^{\mathbb{R}}\leq 2^{\frac{3}{4}}$\textbf{,}
and since \ $\mathbf{q}=\left( \frac{4}{3},\frac{4}{2+\tau },\frac{4}{3-\tau
}\right) $ belongs to the convex hull of $\mathbf{q}\left( 2\right) $ and $%
\mathbf{q}\left( 1\right) $ for $\tau \in \left[ 0,1\right] $, from \
Proposition \ref{up} with $k=2$ we conclude that
\begin{equation*}
C_{3,\mathbf{q}}^{\mathbb{R}}\leq 2^{\frac{3}{4}}.
\end{equation*}%
Thus, if $\mathbf{q}=\left( \frac{4}{3},\frac{4}{2+\tau },\frac{4}{3-\tau }%
\right) $ with $\tau \in \left[ 0,1\right] $, then%
\begin{equation*}
C_{3,\mathbf{q}}^{\mathbb{R}}=2^{\frac{3}{4}}.
\end{equation*}%
This proves the result for $\theta =0$.

Let us prove the case $\theta \in (0,1].$ We can verify that the values $%
\left( \alpha ,\beta ,\gamma \right) \in \left[ 1,2\right] ^{3}$ with $\frac{%
1}{\alpha }+\frac{1}{\beta }+\frac{1}{\gamma }=2$ for which
\begin{equation*}
2^{\frac{2\beta \gamma +\alpha \beta +\alpha \gamma -2\alpha \beta \gamma }{%
\alpha \beta \gamma }}=2^{\frac{\theta +3}{4}}
\end{equation*}%
are $\left( \frac{4}{\theta +3},\frac{4\gamma }{5\gamma -\theta \gamma -4}%
,\gamma \right) ~$ and $\left( \frac{4}{\theta +3},\beta ,\frac{4\beta }{%
5\beta -\theta \beta -4}\right) ~$for $\gamma ,\beta \in \left[ \frac{4}{%
3-\theta },2\right] $ and $\theta \in \left[ 0,1\right] $. Equivalently,
\begin{equation*}
\left( \frac{4}{\theta +3},\frac{4}{2+\tau -\theta \tau },\frac{4}{3+\theta
\tau -\theta -\tau }\right) \text{ for }\tau ,\theta \in \left[ 0,1\right] .
\end{equation*}%

\pagebreak

Thus, we invoke (\ref{tri}) to conclude that for $\overline{\mathbf{q}}%
=\left( \frac{4}{\theta +3},\frac{4}{2+\tau -\theta \tau },\frac{4}{3+\theta
\tau -\theta -\tau }\right) $ with $\tau ,\theta \in \left[ 0,1\right] ,$ we
have%
\begin{equation}
C_{3,\overline{\mathbf{q}}}^{\mathbb{K}}\geq 2^{\frac{\theta +3}{4}}.
\label{27}
\end{equation}

On the other hand, from \cite[Theorem 2.1]{natal} we know that for $%
\overline{\mathbf{q}}\left( 1\right) =\left( 1,2,2\right) ,$ we have $C_{3,%
\overline{\mathbf{q}}\left( 1\right) }^{\mathbb{R}}=2.$ Moreover, from what
we just did we have, for $\tau \in \left[ 0,1\right] $\textbf{\ }and\textbf{%
\ }$\overline{\mathbf{q}}\left( 2\right) =\left( \frac{4}{3},\frac{4}{2+\tau
},\frac{4}{3-\tau }\right) ,$\textbf{\ }that $C_{3,\overline{\mathbf{q}}%
\left( 2\right) }^{\mathbb{R}}=2^{\frac{3}{4}}$. Interpolating $\overline{%
\mathbf{q}}\left( 1\right) =\left( 1,2,2\right) $ and $\overline{\mathbf{q}}%
\left( 2\right) =\left( \frac{4}{3},\frac{4}{3-\tau },\frac{4}{2+\tau }%
\right) $, we obtain
\begin{eqnarray*}
\frac{1}{q_{1}} &=&\frac{\theta }{1}+\frac{1-\theta }{\frac{4}{3}}%
\Rightarrow q_{1}=\frac{4}{\theta +3} \\
\frac{1}{q_{2}} &=&\frac{\theta }{2}+\frac{1-\theta }{\frac{4}{2+\tau }}%
\Rightarrow q_{2}=\frac{4}{2+\tau -\theta \tau }. \\
\frac{1}{q_{3}} &=&\frac{\theta }{2}+\frac{1-\theta }{\frac{4}{3-\tau }}%
\Rightarrow q_{3}=\frac{4}{3+\theta \tau -\tau -\theta }
\end{eqnarray*}%
i.e., $\overline{\mathbf{q}}=\left( \frac{4}{\theta +3},\frac{4}{2+\tau
-\theta \tau },\frac{4}{3+\theta \tau -\theta -\tau }\right) $, and thus
\begin{eqnarray*}
C_{3,\overline{\mathbf{q}}}^{\mathbb{R}} &\leq &2^{\theta }2^{\frac{3}{4}%
\left( 1-\theta \right) } \\
&=&2^{\frac{\theta +3}{4}},
\end{eqnarray*}%
and from $\left( \ref{27}\right) $ we conclude that%
\begin{equation*}
C_{3,\overline{\mathbf{q}}}^{\mathbb{R}}=2^{\frac{\theta +3}{4}}
\end{equation*}%
for $\overline{\mathbf{q}}=\left( \frac{4}{\theta +3},\frac{4}{2+\tau
-\theta \tau },\frac{4}{3+\theta \tau -\theta -\tau }\right) $ with $\tau
\in \left[ 0,1\right] $ and $\theta \in \left[ 0,1\right] .$
\end{proof}

We note that in the above theorem, in order to get the lower estimates, we
have just used (\ref{tri}), which is a consequence of the Theorem \ref%
{general}. If we use Theorem \ref{outro_opera} instead of Theorem \ref%
{general}, we can prove that for $m=3$ and $\mathbf{q}=(\alpha ,\beta
,\gamma )$ we have
\begin{equation*}
C_{3,\mathbf{q}}^{\mathbb{R}}\geq 2^{\frac{2\alpha \gamma +\alpha \beta
+\beta \gamma -2\alpha \beta \gamma }{\alpha \beta \gamma }}.
\end{equation*}%
Using an analogous argument we can prove the following:

\begin{proposition}
\label{210}\bigskip Let $\tau \in \left[ 0,1\right] $ and $\theta \in \left[
0,1\right] $. If
\begin{equation*}
\text{ }\mathbf{q}=\left( \frac{4}{2+\tau -\theta \tau },\frac{4}{\theta +3},%
\frac{4}{3+\theta \tau -\theta -\tau }\right)
\end{equation*}%
then, the optimal constant of the generalized Bohnenblust--Hille inequality
for real scalars is
\begin{equation*}
C_{3,\mathbf{q}}^{\mathbb{R}}=2^{\frac{\theta +3}{4}}.
\end{equation*}
\end{proposition}

\begin{center}
\begin{figure}[tbh]
\includegraphics[scale=0.4]{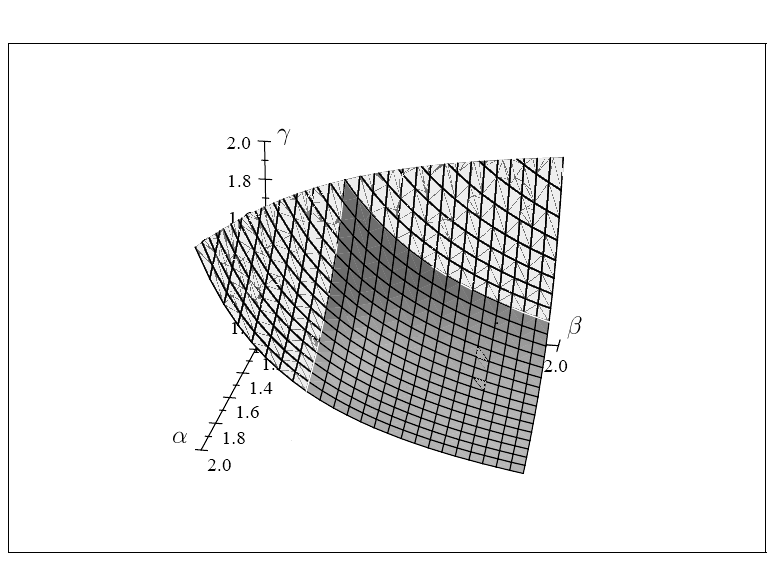}
\caption{The $3$-linear case: The graph, whose origin is $(1,1,1)$,  represents points such that $%
\left( \protect\alpha ,\protect\beta ,\protect\gamma \right) \in \left[ 1,2%
\right] ^{3}$ with $\frac{1}{\protect\alpha }+\frac{1}{\protect\beta }+\frac{%
1}{\protect\gamma }=2$. In the less dark region we have the points where the optimality was proved (Proposition \protect\ref{110} and Proposition \protect\ref{210})}
\end{figure}
\end{center}

\pagebreak

\end{document}